\documentclass[10pt,a4paper]{amsart}
\usepackage[latin1]{inputenc}
\usepackage{amsmath,enumerate,amsthm,graphicx,color,verbatim}
\usepackage{amsfonts,amscd}
\usepackage{latexsym}
\usepackage{amssymb}
\newtheorem{thm}{Theorem}
\newtheorem{defn}{Definition}
\newtheorem*{note}{Note}

\newtheorem{lemma}{Lemma}
\newtheorem{cor}{Corollary}
\newtheorem{prop}{Proposition}

\newcommand{\Z}{\mathbb{Z}}
\newcommand{\D}{\mathbb{D}}
\newcommand{\R}{\mathbb{R}}

\newcommand{\Q}{\mathbb{Q}}

\newcommand{\norm}{\vert\vert}

\begin{document}
\title[Limit-Periodic Verblunsky Coefficients]{Limit-Periodic Verblunsky Coefficients for Orthogonal Polynomials on the Unit Circle}

\author[D.\ Ong]{Darren C. Ong}

\address{Department of Mathematics, Rice University, Houston, TX~77005, USA}

\email{darren.ong@rice.edu}

\urladdr{http://math.rice.edu/$\sim$do3}

\thanks{Supported in part by NSF grant
DMS-1067988.}
\maketitle
\begin{abstract}
Avila recently introduced a new method for the study of the discrete
Schr\"odinger Operator with limit periodic potential. I adapt this method to
the context of orthogonal polynomials in the unit circle with limit periodic
Verblunsky Coefficients. Specifically, I represent these two-sided Verblunsky
Coefficients as a continuous sampling of the orbits of a Cantor group by a
minimal translation. I then investigate the measures that arise on the unit
circle as I vary the sampling function. I show that generically the spectrum is a Cantor set and we have empty point spectrum. Furthermore, there exists a dense set of sampling functions for which the corresponding spectrum is a Cantor set of positive Lebesgue measure, and all corresponding spectral measures are purely absolutely continuous.
\end{abstract}
{\textbf{Keywords}} Spectral Theory, Orthogonal Polynomials on the Unit Circle, Limit-periodicity.

\begin{section}{Introduction}
{}
There is a strong relationship between orthogonal polynomials corresponding to probability measures on the unit circle (OPUC), and orthogonal polynomials corresponding to probability measures on the real line (OPRL). This connection is very carefully explored in \cite{Simon}. In particular, we are often able to adapt results concerning the spectral measure of the discrete Schr\"odinger Operator to the context of probability measures on the unit circle.

{}
In this paper, we adapt the results of \cite{Damanik-Gan} to the unit circle context in this way. In other words, we examine the properties of the probability measure on the unit circle when the recurrence coefficients of the associated orthogonal polynomials (the \em Verblunsky coefficients\em) are a limit-periodic sequence. I wish here to thank Damanik and Gan, the authors of \cite{Damanik-Gan} for many helpful suggestions. I wish also to thank the anonymous referee.

\begin{defn}
A \em limit-periodic \em sequence in $\ell^\infty(\Z)$ is a sequence that is expressed as a uniform limit of periodic sequences in $\ell^\infty(\Z)$.
\end{defn}
{}
The main thrust of our argument relies on an adaptation of a method introduced in \cite{Avila} and utilized heavily in \cite{Damanik-Gan}, whereby limit periodic potential sequences of the discrete Schr\"odinger Operator are expressed as a continuous sampling of a minimal translation on a Cantor group.

\begin{defn}
A Cantor group $\Omega$ is a totally disconnected compact Abelian topological group with no isolated points. We say that a map $T:\Omega\to\Omega$ is a translation if it is of the form $T\omega=\omega+a$ for some $a\in\Omega$. A translation $T$ is called minimal if the $T$-orbit of every $\omega\in\Omega$ is dense in $\Omega$.
\end{defn}
{}
Our setup is as follows: we fix $\Omega$ and $T$ throughout, where $\Omega$ is a Cantor group, and $T:\Omega\to \Omega$ a minimal translation. Then, a sampling function $f\in C(\Omega, \D)$ induces a sequence of Verblunsky coefficients $\alpha_j=f(T^j\omega)$, for some $\omega\in \Omega$. We shall denote the spectrum of $\mathcal E_f$, the induced extended CMV operator (The definition of the extended CMV matrix can be found in (\cite{Simon}, 10.5.34 and 10.5.35)) as $\Sigma(\mathcal E_f)$. By minimality, this spectrum is independent of $\omega$ (by Theorem 10.16.1 of \cite{Simon}). The extended CMV matrix is a $5$-diagonal unitary matrix composed of $2\times 4$ blocks, written as

\begin{equation}\label{CMVmatrix}
\mathcal E=\left(
\begin{array}{ccccccc}
\ldots&\ldots& \ldots&\ldots&\ldots&\ldots&\ldots\\
\ldots&-\bar\alpha_0\alpha_{-1}&\bar \alpha_1\rho_0&\rho_1\rho_0&0&0&\ldots\\
\ldots&-\rho_0\alpha_{-1}&- \bar\alpha_1\alpha_0&-\rho_1\alpha_0&0&0&\ldots\\
\ldots&0& \bar\alpha_2\rho_1&-\bar\alpha_2\alpha_1&\bar\alpha_3\rho_2&\rho_3\rho_2&\ldots\\
\ldots&0& \rho_2\rho_1&-\rho_2\alpha_1&-\bar\alpha_3\alpha_2&-\rho_3\alpha_2&\ldots\\
\ldots&0& 0&0&\bar\alpha_4\rho_3&-\bar\alpha_4\alpha_3&\ldots\\
\ldots&\ldots& \ldots&\ldots&\ldots&\ldots&\ldots\\
\end{array}
\right),
\end{equation}
{}
where $\rho_n=\sqrt{1-\vert\alpha_n\vert^2}$. We now are able to state our three main theorems:

\begin{thm}\label{t.gordon}
There exists a dense $G_{\delta}$ set $\mathcal G\subseteq C(\Omega,\D)$ such that for every $f\in\mathcal G$ and $\omega\in \Omega$, the extended CMV operator, $\mathcal E_f$ has empty point spectrum. 
\end{thm}
\begin{thm}\label{t.cantor}
There exists a dense $G_\delta$ set $\mathcal R\subseteq C(\Omega,\D)$ such that for every $f\in\mathcal R$ and $\omega\in \Omega$, the spectrum of the extended CMV operator, $\mathcal E_f$ is a Cantor set.
\end{thm}
\begin{thm}\label{t.ac}
There exists a dense set $\mathcal R'\subseteq C(\Omega,\D)$ such that for every $f\in\mathcal R'$ and $\omega\in \Omega$, the spectrum of the extended CMV operator, $\mathcal E$ is a Cantor set of positive Lebesgue measure and all spectral measures are purely absolutely continuous.
\end{thm}

{}
In short, we prove an OPUC version of \cite{Damanik-Gan}, except for the singular continuous spectrum results proved in \cite{Avila}. An initial attempt to adapt the results of \cite{Avila} was unsuccessful, due to the difficulty of describing and proving an OPUC analogue of Lemma 2.3 in that paper. However, I believe that the analogous singular continuous spectrum result for OPUC is true nevertheless: in other words, that for a dense $G_\delta$ set in $C(\Omega,\D)$ the corresponding spectrum is purely singular continuous.

{}
To prove these theorems, we will rely on the fact that limit-periodic sequences are very well approximated by periodic sequences. The following remarks about periodic sampling functions will prove essential:

\begin{defn}
A sampling function $f$ is $k$-periodic if $f(T^k\omega)=f(\omega)$ for every $\omega\in\Omega$. The space of periodic sampling functions is denoted $P$. 
\end{defn}
We note that Lemma 2.4 in \cite{Damanik-Gan} applies for our context. This lemma states
\begin{prop}\label{cantorsubgroups}
There exists a sequence of Cantor subgroups $\Omega_k$ of $\Omega$ of finite index $p_k$ such that $\bigcap \Omega_k=\{0\}$. The sampling functions that are defined on $\Omega/\Omega_k$ are periodic with period $p_k$. We denote those sampling functions as the space $P_k\subset P$. All periodic sampling functions belong to some $P_k$.
\end{prop}

\end{section}
\begin{section}{Absence of Point Spectrum}
{}
In this section we prove Theorem \ref{t.gordon}. Here we adapt the Gordon potential idea presented in \cite{Gordon} to the OPUC case, specifically the treatment given in \cite{Cycon}.

{}
We will begin by presenting an explanation of transfer matrices related to $\mathcal E$. Let us first define an \em$\ell^\infty$-eigenvector \em of $\mathcal E$ as an $\ell^\infty$ vector $u$ that satisfies $zu=\mathcal E u$ for some $z$, but such that $u$ is not necessarily in $\ell^2$. Based on the structure given in (\ref{CMVmatrix}), we determine that for $n$ odd, an $\ell^{\infty}$-eigenvector $u$ of $\mathcal E $ must satisfy
\begin{align}
zu_{n+1}=&\bar\alpha_{n+1}\rho_n u_n-\bar \alpha_{n+1}\alpha_n u_{n+1}+\bar \alpha_{n+2}\rho_{n+1} u_{n+2}+\rho_{n+2}\rho_{n+1} u_{n+3}\label{CMV1},\\ 
zu_{n+2}=&\rho_{n+1}\rho_n u_n-\rho_{n+1}\alpha_n u_{n+1}-\bar \alpha_{n+2}\alpha_{n+1} u_{n+2}-\rho_{n+2}\alpha_{n+1} u_{n+3}\label{CMV2}.
\end{align}
We multiply Equation (\ref{CMV1}) by $\alpha_{n+1}$ and Equation (\ref{CMV2}) by $\rho_{n+1}$ and then add the two equations to eliminate the $u_{n+3}$ term. This gets us
\[ z\rho_{n+1}u_{n+2}=(\bar\alpha_{n+1}\rho_n\alpha_{n+1}+\rho_{n+1}^2\rho_n)u_n-((\bar \alpha_{n+1}\alpha_n+z)\alpha_{n+1}+\rho_{n+1}^2\alpha_n)u_{n+1}.\]
Similarly, if we multiply Equation (\ref{CMV1}) by $\bar \alpha_{n+2}\alpha_{n+1}+z$ and Equation (\ref{CMV2}) by $\bar \alpha_{n+2}\rho_{n+1}$, adding the two equations eliminates the $u_{n+2}$ term and we have
\begin{align*}
-z\rho_{n+2}\rho_{n+1}u_{n+3}=&(\bar\alpha_{n+1}\rho_n (\bar \alpha_{n+2}\alpha_{n+1}+z)+\rho_{n+1}^2\rho_n \bar \alpha_{n+2})u_n\\&-((\bar \alpha_{n+1}\alpha_n+z)(\bar \alpha_{n+2}\alpha_{n+1}+z)+\rho_{n+1}^2\alpha_n\bar \alpha_{n+2}) u_{n+1}.
\end{align*}
These equations give us a two-step transfer matrix
\[
A_n=\frac{1}{z\rho_{n+1}\rho_{n+2}}
\left(
\begin{array}{cc}
a_{11}&a_{12}\\
a_{21}&a_{22}
\end{array}
\right),
\]
with entries
\begin{align*}
a_{11}=&\rho_{n+2}(\bar\alpha_{n+1}\alpha_{n+1}\rho_n+\rho_{n+1}^2\rho_n),\\
a_{12}=&-\rho_{n+2}[(\bar\alpha_{n+1}\alpha_n+z)\alpha_{n+1}+\rho_{n+1}^2\alpha_n],\\
a_{21}=&-[(\bar\alpha_{n+2}\alpha_{n+1}+z)\bar\alpha_{n+1}\rho_n+\rho_{n+1}^2\rho_n\bar \alpha_{n+2}],\\
a_{22}=& (\bar \alpha_{n+2}\alpha_{n+1}+z)(\bar \alpha_{n+1}\alpha_n+z)+\bar \alpha_{n+2}\rho_{n+1}^2\alpha_{n}.
\end{align*}
This satisfies
\begin{equation}\label{recurrence}
\left(
\begin{array}{c}
u_{n+2}\\
u_{n+3}
\end{array}
\right)=
A_n
\left(
\begin{array}{c}
u_n\\
u_{n+1}
\end{array}
\right).
\end{equation}
We can compute that $\det A_n=\rho_n/\rho_{n+2}$, which is never zero. Hence $A_n$ is always invertible.

\begin{note}
It is possible to express the recurrence relation so that the transfer matrices have determinant $1$. Let $\mathfrak A_n$ be the matrix that results from multiplying the top row of $A_n$ by $\rho_{n+2}$ and the left column of $A_n$ by $1/\rho_n$. Consequently, we can express Equation (\ref{recurrence}) instead as 
\begin{equation}
\left(
\begin{array}{c}
\rho_{n+2}u_{n+2}\\
u_{n+3}
\end{array}
\right)=
\mathfrak A_n
\left(
\begin{array}{c}
\rho_nu_n\\
u_{n+1}
\end{array}
\right),
\end{equation}
where this time the matrix $\mathfrak A_n$ has determinant $1$. This adjustment is based on a similar trick used to obtain determinant $1$ transfer matrices in the OPRL case. We choose not to use this expression for the recurrence relation, since it is not essential for our context that the transfer matrices have determinant $1$, and the $\rho_n$ terms complicate the proof slightly.
\end{note}
{}
Let $\mathbb D(r)$ be an open disk of radius $r<1$ centered at the origin. Note that $a_{11}, a_{12}, a_{21}, a_{22}$ (as well as the fraction $1/z\rho_{n+1}\rho_{n+2}$) are all continuous functions of $\alpha_n,\alpha_{n+1},\alpha_{n+2}$. In addition, if we restrict $\alpha_n,\alpha_{n+1}, \alpha_{n+2}$ to $\mathbb D(r)$, $a_{11}, a_{12}, a_{21}, a_{22}$ and  $1/z\rho_{n+1}\rho_{n+2}$ are also bounded. There must exist a positive real-valued function $\gamma(k,q,r)$ with $k,q\in \mathbb Z_+$ so that if $\{\tilde \alpha_i\}$ are a sequence of Verblunsky coefficients corresponding to the matrices $\tilde A_n$, and if $\vert\tilde\alpha_n-\alpha_n\vert, \vert\tilde\alpha_{n+1}-\alpha_{n+1}\vert,\vert\tilde\alpha_{n+2}-\alpha_{n+2}\vert$ are all less than $\gamma(k,q,r)$, this implies that $\norm A_n-\tilde A_n\norm< k^{-q}$.

\begin{prop}\label{Gordon}Let $\alpha(n)$ and $\alpha_k(n)$, $k\in \mathbb Z_+$, be two sided-sequences in $\D^{\infty}$. Let  $\alpha_k$ be periodic with even period $q_k$. We then know that there must exist $r_k<1$ so that the sequence $\alpha_k(n)$ lies in $\mathbb D(r_k)$, the open disk of radius $r_k$ centered at the origin. If
 \[\sup_{-2q_k+1\leq n\leq 2q_k+1}\vert \alpha_k(n)-\alpha(n)\vert\leq \gamma(k,q_k,r_k),\]
and $\mathcal E$ is the extended CMV operator corresponding to Verblunsky coefficients $\alpha(n)$, there is no eigenvector $u$ of $\mathcal E$ in $\ell^2(\Z)$.
\end{prop}
We first introduce the following lemma:
\begin{lemma}\label{matrix}
Let $A$ be an invertible $2\times 2$ matrix, and $x$ a vector of norm $1$. Then
\[ \max(\norm Ax\norm, \norm A^2 x\norm, \norm A^{-1} x\norm, \norm A^{-2} x\norm)\geq 1/2.\]
\end{lemma}
\begin{proof}
This proof is in Section 10.2 of \cite{Cycon}. 
\end{proof}
\begin{note}
This is known as the ``four block'' version of the Gordon idea. We don't use the ``three block'' version of the idea since we don't want to assume our matrices have determinant $1$.
\end{note}
\begin{proof}[Proof of Proposition]
Let $\mathcal E_k$ be the extended CMV operator corresponding to the Verblunsky coefficients $\alpha_k(n)$. Let $w^{(k)}$ be a nonzero $\ell^\infty$-eigenvector of $\mathcal E_k$ with the same initial condition as the nonzero $\ell^\infty$-eigenvector $u$ of $\mathcal E$: i.e.  $u_1=w^{(k)}_1,u_2=w^{(k)}_2$.  Note that $u_1, u_2$ cannot both be zero, since otherwise the recurrence would imply that $u$ is zero. Define for odd $n$
\[ 
\phi(n)=\left(
\begin{array}{c}
u_n\\
u_{n+1}
\end{array}
\right),
\phi_k(n)=\left(
\begin{array}{c}
w^{(k)}_n\\
w^{(k)}_{n+1}
\end{array}
\right).
\]
and $A_n^{(k)}$, as the transfer matrix of $\mathcal E_k$. Assume for now that the odd number $n$ is positive.
In the following calculations, bear in mind that all the suprema only run through odd $n$, and the sums only run through odd $j$. Also, let $\Phi=\norm \phi(1)\norm$.
\begin{align*}
\sup_{ 1\leq n\leq 2q_k+1} &\norm \phi(n)-\phi_k(n)\norm\\
\leq \sup_{ 1\leq n\leq 2q_k+1} &\norm  A_{n-2}A_{n-4}\ldots A_3 A_1-A_{n-2}^{(k)}A_{n-2}^{(k)}\ldots A_3^{(k)}A_1^{(k)}\norm\Phi,\\
= \sup_{ 1\leq n\leq 2q_k+1}& \left\vert\left\vert \sum_{1\leq j\leq n-2} A_{n-2}\ldots A_j A_{j-2}^{(k)}\ldots A_1^{(k)}-A_{n-2}\ldots A_{j+2} A_{j}^{(k)}\ldots A_1^{(k)}\right\vert\right\vert\Phi,\\
 \leq \sup_{ 1\leq n\leq 2q_k+1}& \sum_{1\leq j\leq n-2}\norm  A_{n-2}\ldots A_{j+2}(A_j-A_j^{(k)})  A_{j-2}^{(k)}\ldots A_1^{(k)}\norm\Phi,\\
 \leq \sup_{ 1\leq n\leq 2q_k+1} &\left(\sum_{1\leq j\leq n-2}\norm  A_{n-2}\ldots A_{j+2}\norm \cdot\norm A_j-A_j^{(k)}\norm  \cdot\norm A_{j-2}^{(k)}\ldots A_1^{(k)}\norm\right)\Phi.\\
\end{align*}

 The entries of the transfer matrices are all bounded. Note also that since \\ $\sup_{n\leq 2q_k+1} \vert \alpha(n)-\alpha_k(n)\vert\leq \gamma(k,q_k,r_k)$, it is true that $\norm A_j-A_j^{(k)}\norm\leq k^{-q_k}$. This means that for some constants $C,D$ we have the estimate
\begin{align*}
\sup_{ 1\leq n\leq 2q_k+1, n \text{ odd}} \norm \phi_k(n)-\phi(n)\norm\leq& \sup_{ 1\leq n\leq 2q_k+1, n \text{ odd}} C\vert n\vert e^{\vert Dn
\vert} k^{-q_k},\\
=& C(2q_k+1)e^{D(2q_k+1)} k^{-q_k}.
\end{align*}
If instead $n$ is negative, the preceding argument proceeds completely analogously.\\

We have now that 
\[\max_{a=\pm 1,\pm 2}\norm \phi(aq_k+1)-\phi_k (aq_k+1)\norm \to 0.\]
as $k\to\infty$. Let us set $A=A^{(k)}_{q_k-1}\ldots A^{(k)}_3A^{(k)}_1$. Note that, due to the fact that the $\{\alpha^{(k)}_i\}$ are $q_k$-periodic, we then have 
\begin{align*}
\phi_k(2q_k+1)=&A^2\left ( 
\begin{array}{c}
u_1\\
u_2
\end{array}
\right), \\
\phi_k(q_k+1)=&A\left ( 
\begin{array}{c}
u_1\\
u_2
\end{array}
\right), \\
\phi_k(-q_k+1)=&A^{-1}\left ( 
\begin{array}{c}
u_1\\
u_2
\end{array}
\right), \\
\phi_k(-2q_k+1)=&A^{-2}\left ( 
\begin{array}{c}
u_1\\
u_2
\end{array}
\right).
\end{align*}
Thus by applying Lemma \ref{matrix} we obtain
\[\max_{a=\pm 1,\pm 2}\norm \phi_k(aq_k+1)\norm \geq \frac{1}{2}\norm \phi_k(1)\norm= \frac{\sqrt{\vert u_1\vert^2+\vert u_2\vert^2}}{2}.\]
Then
\begin{align*}
\limsup_{n \text{ odd integer }}\frac{\vert u_n\vert^2+\vert u_{n+1}\vert^2}{\vert u_1\vert^2+\vert u_2\vert^2}\geq\limsup_{q_k} \frac{\max_{a=\pm 1,\pm 2}\norm \phi_k(aq_k+1)\norm^2}{\norm \phi(1)\norm^2}\geq \frac{1}{4}.
\end{align*}
\end{proof}
{}In particular, a CMV operator associated with Verblunsky coefficients that satisfy the conditions of Proposition \ref{Gordon} must have empty point spectrum. 

\begin{cor}\label{c.gordon}
Consider a two-sided sequence of Verblunsky coefficients $\alpha(n)$ such that there exists a sequence of even positive integers $q_k\to \infty$ so that 
\[
\max_{-q_k+1\leq n\leq q_k+1}\vert \alpha(n)-\alpha(n\pm q_k)\vert\leq \frac{\gamma(k, q_k,r_k)}{4}.
\]
Here, $r_k$ is a positive real number less than $1$ for which $\alpha(-2q_k+1),\ldots, \alpha(2q_k+1)$ all lie in a disk of radius $r_k$ centered at the origin.
The CMV operator associated with these Verblunsky coefficients has no point spectrum.
\end{cor}
\begin{defn}
A two-sided sequence of Verblunsky coefficients $\alpha(n)\in \D$ that satisfies the conditions of Corollary \ref{c.gordon} is referred to as a \em Gordon sequence\em . 
\end{defn}

\begin{proof}[Proof of Theorem \ref{t.gordon}]
The set of periodic sampling functions is dense in $C(\Omega,\D)$. For every $j$-periodic sampling function $f_j$, choose a positive real number $R(f_j)<1$ so that $\norm f_j\norm< R(f_j)$. Now for $j\in 2\Z_+,k\in\Z_+$, let 
\begin{align*}
\mathcal G_{j,k}=&\{f\in C(\Omega, \D):\text{ $ \exists j$-periodic $f_j$, such that } \norm f\norm <R(f_j),  \\
&\text{ and }\vert f(v)-f_j(v)\vert<\frac{1}{2}\left(\frac{\gamma(k,jk,R(f_j))}{4}\right), v\in \Z\},
\end{align*}

Clearly, $\mathcal G_{j,k}$ is open. For $k\in\Z_+$, let
\[\mathcal G_k=\bigcup_{j=2,j \text{ even}}^\infty \mathcal G_{j,k}.
\]
The set $\mathcal G_k$ is open by construction, and dense since it contains all periodic sampling functions (since we may double the period of every odd-periodic sampling function). Thus 
\[
\mathcal G=\bigcap_{k=1}^\infty \mathcal G_k.
\]
is a dense $G_\delta$ set. We claim that for every $f\in \mathcal G$ and every $\omega\in \Omega$, the sequence of Verblunsky coefficients given by $\alpha(n)=f(T^n\omega)$ is Gordon. \\

Let $f\in \mathcal G$ and $\omega\in\Omega$ be given. Since $f\in \mathcal G_k$ for every $k$, we can find $j_k$-periodic $f_{j_k}$ satisfying 
\[\vert f(v)-f_{j_k}(v)\vert <\frac{1}{2}\left(\frac{\gamma(k,j_kk,R(f_{j_k}))}{4}\right), v\in \Z.\]

Let $q_k = j_k k$, so that $q_k \rightarrow \infty$ as $k
\rightarrow \infty$. Then, we have
\begin{align*}
&\max_{-q_k+1\leq n \leq q_k+1}  \|\alpha (n) - \alpha (n \pm q_k)\|\\
=& \max_{-q_k+1\leq n \leq q_k+1} \|f(T^n \omega) - f(T^{n \pm
q_k} \omega)\|, \\
 =& \max_{-j_kk+1 \leq n
\leq j_k k+1} \|f(T^n\omega) - f_{j_k}(T^n\omega) + f_{j_k}(T^{n \pm j_k k} \omega) - f(T^{n \pm j_k k} \omega)\|, \\
 \leq& \max_{-j_kk+1 \leq n \leq j_k k+1} \|f(T^n \omega) - f_{j_k}(T^n
\omega)\| \\
&+ \max_{-j_kk+1 \leq n \leq j_k k+1} \|f(T^{n \pm j_k k} \omega)
- f_{j_k}(T^{n \pm j_k k} \omega)\|, \\
 <&\frac{\gamma(k,q_k, R(f_{j_k}))}{4}.
\end{align*}
We may, of course simply define $r_k=R(f_{j_k})$. It follows that $\alpha$ is a Gordon sequence.
\end{proof}
\end{section}
\begin{section}{Cantor spectra}
{}
In this section we will prove Theorem \ref{t.cantor}. First, we establish a connection between perturbations of the sampling function and perturbations of the corresponding CMV spectra.
\begin{prop}\label{72}Given two sampling functions $f,g$ with $\norm f-g\norm_\infty<\epsilon^2/72$ for some $\epsilon>0$, and $\mathcal E_f, \mathcal E_g$ denoting the extended CMV operators induced by $f,g$ respectively, we have
\[\norm \mathcal E_{f}-\mathcal E_g\norm \leq \epsilon.\]
\end{prop}
\begin{proof}
This follows immediately from Equation (4.3.11) in \cite{Simon}.
Note that Equation (4.3.11) states a result for regular CMV matrices. The proof works perfectly well for our extended CMV matrices. 
\end{proof}
Let us first introduce the following elementary lemma:
\begin{lemma}\label{elementary}
Let $\mathcal E$ be a unitary operator. Let $z$ be a point on $\partial \D$, and let $d$ be the distance from $z$ to the spectrum $\Sigma(\mathcal E)$ of $\mathcal E$ (measured in absolute distance, rather than arclength). Then where $R_z$ is the resolvent of $\mathcal E$ at $z$, $\norm R_z\norm = 1/d$.
\end{lemma}
\begin{proof}
Let $\psi$ denote an arbitrary normalized vector. Let $\mu$ be the spectral measure of $\mathcal E$ corresponding to $\psi$. Since $\mu$ is supported by $\Sigma(\mathcal{E})$, by the Spectral Theorem,
\[ \norm R_z\psi\norm^2=\int_{\Sigma(\mathcal E)}\frac{1}{\vert w-z\vert^2}d\mu(w)\leq\frac{1}{d^2}. \]
Thus we have demonstrated that $\norm R_z\norm \leq 1/d$. The assertion that $\norm R_z\norm \geq 1/d$ is given in (2.88) in \cite{Teschl}.
\end{proof}
\begin{prop}\label{rotation}
Given two CMV operators $\mathcal E, \mathcal E'$, if $\norm \mathcal E-\mathcal E'\norm< \epsilon$, then if $z_0$ is in $\mathcal E$'s spectrum, then $\vert z_0-z_0'\vert<\epsilon$ for some $z_0'$ in $\mathcal E'$'s spectrum.
\end{prop}
\begin{proof}
By Lemma 2.16 in \cite{Teschl}, if $z_0\in \partial \D$ and there exists a sequence 
$\{f_n\}\in \ell^2(\Z)$, $f_n\neq 0$ so
\[
\frac{\norm (\mathcal E-z_0)f_n\norm}{\norm f_n \norm}\to 0,n\to\infty,
\]
then $z_0$ is in the spectrum of $\mathcal E$. If $z_0$ is in the boundary of the resolvent set, then the converse also holds.

We apply the converse. So let $\mathcal E, \mathcal E'$ be two CMV operators such that \\ $\norm \mathcal E-\mathcal E'\norm<\epsilon$ for some $\epsilon>0$. Let $z_0$ be in $\Sigma(\mathcal E)$, and let $z_0'$ be a point in $\Sigma(\mathcal E')$ whose distance from $z_0$ is minimum ($z_0'$ exists since the spectrum is closed). If $z_0=z_0'$, we're done. If not, it is clear that $z_0'$ is on the boundary of the resolvent set of $\mathcal E'$. Hence we know that there exists a sequence $\{f_n\}$ for which
\[ \frac{\norm (\mathcal E'-z_0') f_n\norm}{\norm f_n\norm}\to 0.\]
But then we have
\begin{align*}
\frac{\norm (\mathcal E-z_0') f_n\norm}{\norm f_n\norm}\leq & \frac{\norm (\mathcal E-\mathcal E') f_n\norm+\norm (\mathcal E'-z_0')f_n\norm }{\norm f_n\norm}.\\
\end{align*}
Thus given any $\delta>0$,
\[e_n=\frac{\norm (\mathcal E-z_0') f_n\norm}{\norm f_n\norm}\leq \epsilon+\delta,\]
for sufficiently large $n$. We know that $e_n$ is a bounded sequence, and thus must have a convergent subsequence. That subsequence must converge to a value not greater than $\epsilon$. Let $\{e_{k(n)}\}_{n=-\infty}^\infty$ be that subsequence, and let it converge to $\gamma\leq \epsilon$ as $n\to\infty$. \\

But then we know that 
\begin{align*}
1=&\frac{\norm(\mathcal E-z_0')^{-1}(\mathcal E-z_0')f_{k(n)}\norm}{\norm f_{k(n)}\norm},\\
\leq &\norm(\mathcal E-z_0')^{-1}\norm \cdot\frac{\norm(\mathcal E-z_0')f_{k(n)}\norm}{\norm f_{k(n)}\norm},\\
\end{align*}
this then implies that
\begin{align*}
\frac{1}{\gamma}\leq&\norm(\mathcal E-z_0')^{-1}\norm=\norm \text{Resolvent function of }\mathcal E \text{ at }z_0'\norm.
\end{align*}
Thus the distance of $z_0'$ from the spectrum of $\mathcal E$ is at most $\gamma\leq \epsilon$, by Lemma \ref{elementary}.
\end{proof}
\begin{lemma}\label{emptyinterior}
Let $\mathcal E_f$ denote the CMV operator induced by the sampling function $f$. Then $\mathcal R=\{f\in C(\Omega,\D):\Sigma(\mathcal E_f)$ has empty interior $\}$ is a $G_{\delta}$ set.
\end{lemma}
\begin{proof}
This proof is similar in spirit to Lemma 1.1 in \cite{Avron-Simon}. For $a,b$ real, we have $e^{ia},e^{ib}\in\partial \D$ and let us define
\begin{align*}
S_{(a,b)}=\{ &f\vert f\in C(\Omega, \D), \text{ the resolvent set of }\mathcal E_f\text{ intersects the open }\\
&\text{counterclockwise  arc between $e^{ia}$ and $e^{ib}$}\}.
\end{align*}
Then
\[\mathcal R=\bigcap_{a,b\in 2\pi\Q} S_{(a,b)},\]
so it suffices to show that $S_{(a,b)}$ is open for every choice of $a,b\in 2\pi\Q$. Let $f\in S_{(a,b)}$. Let $B$ be the open counterclockwise arc from $e^{ia}$ and $e^{ib}$ Since the resolvent set of an operator is always open, there must exist $c\in\partial\D,\delta\in\mathbb \R_+ $ so that the open counterclockwise arc $A$ centered at $c$ with endpoints each $\delta$ away from $c$ in absolute distance satisfies
\[A\subseteq B\cap\text{ resolvent set of } \mathcal E_f.\]
For every $g\in C(\Omega, \D)$ with  $\norm f-g\norm<\delta^2/72$, we assert that $c$ is in the resolvent set of $\mathcal E_g$. This is because we have $\norm \mathcal E_f-c\norm^{-1}\leq\delta^{-1}$ by Lemma \ref{elementary}, so that 
\[1+(\mathcal E_g-\mathcal E_f)(\mathcal E_f-c)^{-1},\]
is invertible if $\norm f-g\norm <\delta^2/72$, since Proposition \ref{72} would then imply that \[\norm \mathcal E_f-\mathcal E_g\norm<\delta.\]
\end{proof}
{}
We will now  say a few words first about the spectrum when the sampling function $f$ (and hence $\alpha$) is periodic.  This case is very extensively studied in Chapter 11 of \cite{Simon}, but we will state a few results that are of particular relevance to us here. 

{}
Given a $p$-periodic sequence of Verblunsky coefficients with $p$ even, we introduce the discriminant function, $\Delta(z)$, which is defined in (11.1.2) of \cite{Simon} and whose properties are given by Theorem 11.1.1 of \cite{Simon}. As a summary, on $\partial \D$, $\Delta$ is continuous and takes on real values. According to Theorem  11.1.2 of \cite{Simon} the essential support of the absolutely continuous spectrum corresponding to the periodic Verblunsky coefficients is given by 
\[\{e^{i\theta} \vert -2\leq\Delta(e^{i\theta})\leq 2\}.\]
This set comprises of $p$ closed bands on $\partial\D$ on each of which $\Delta$ as a function of $\theta$ is either strictly increasing or strictly decreasing (Figure 11.2 of \cite{Simon} provides a typical picture of the graph of $\Delta$). Two adjacent bands intersect not at all or at just one point. If they do not intersect, we say that the gap between two adjacent bands is \em open.\em 

\begin{lemma}\label{opengaps}For a dense open set of periodic samping functions $f\in P$, all the gaps of $\Sigma( \mathcal E_f)$ are open.
\end{lemma}
\begin{proof}
Theorem 11.13.1 in \cite{Simon} tells us that a the set of sequences in $\D^p$ which represent the first $p$ terms of a $p$-periodic sequence corresponding to a spectrum with at least one gap closed is a closed set of measure zero. In particular, if we have a $p$-periodic Verblunsky sequence whose first $p$ terms is $\alpha_0,\ldots \alpha_{p-1}$, and that sequence corresponds to a spectrum with at least one closed gap, we can find another $p$-periodic Verblunsky coefficient sequence starting with $\tilde \alpha_0,\ldots \tilde\alpha_{p-1}$ with $\max_{1\leq j\leq p-1}\vert \alpha_j-\tilde\alpha_j\vert$ as small as we like,  which corresponds to a spectrum with all gaps open instead.\\

Take an $\epsilon>0$. Now consider a $p$-periodic sampling function $f\in P$ that corresponds to a spectrum with at least one closed gap. By the $\omega$-independence of the spectrum, it suffices to choose an $\omega_0\in \Omega$, and let $\alpha_j=f(T^j\omega_0)$ for $0\leq j\leq p-1$. We can find $\tilde \alpha_0,\ldots \tilde\alpha_{p-1}$ such that $\vert \alpha_j-\tilde\alpha_j\vert<\epsilon$ so that a $p$-periodic sequence starting with $\tilde \alpha_0,\ldots \tilde\alpha_{p-1}$ corresponds to a spectrum with all gaps open. \\

By Proposition \ref{cantorsubgroups}, we know that $p$-periodic functions are defined on $\Omega/\Omega_k$ for some $k$, where $\Omega_k$ is an index $p$ subgroup of $\Omega$. In particular, $\alpha_0,\ldots \alpha_{p-1}$ are the images of $f$ on the $p$ cosets of $\Omega_k$ in $\Omega$. We define $\tilde f$ by replacing the images of those $p$ cosets with $\tilde\alpha_0,\ldots \tilde\alpha_{p-1}$. In this case, $\tilde f$ is clearly a continuous $p$-periodic sampling function on $\Omega$, for which $\norm f-\tilde f\norm <\epsilon$ and $\tilde f$ corresponds to a spectrum with all gaps open.
\end{proof}
\begin{lemma}\label{smallbands}
Let $f\in P$ have period $p$. Then the measure of each band of $\Sigma(\mathcal E_f)$ is at most $2\pi/p$.
\end{lemma}
\begin{proof}
This follows from Theorem 11.1.3 in \cite{Simon}, in particular (11.1.18) and (11.1.21). The first part of that Theorem 11.1.3 tells us that the essential spectrum of the spectral measure is equal to the support of the equilbrium measure $d\nu$ on the bands, and (11.1.18) and (11.1.21) together state that the equilbrium measure of any single band is $1/p$. 

Now consider the estimate of the equilbrium measure given in Theorem 10.11.21 of \cite{Simon}. Periodic Verblunsky coefficients are a special case of stochastic Verblunsky coefficients, so this theorem applies. We have then that we can write the equilbrium measure as $d\nu(\theta)=g(\theta) \frac{d\theta}{2\pi}$, with $g(\theta)>1$ for almost every $\theta$ in the support. But this then implies that the normalized Lebesgue measure of any band is at most $1/p$. Hence the Lebesgue measure of any band is at most $2\pi/p$.
\end{proof}

\begin{proof}[Proof of Theorem \ref{t.cantor}]
We need to show that the subset $\mathcal R$ of $C(\Omega,\D)$ defined in Lemma \ref{emptyinterior} is dense. Since $P$ is dense in $C(\Omega, \mathbb D)$, we need only show that given $f\in P$ and $\epsilon>0$, there is an $\tilde f$ such that $\norm f-\tilde f\norm<\epsilon$ and $\Sigma(\mathcal E_{\tilde f})$ is nowhere dense.\\

Let $P$ be the set of periodic sampling functions, and $P_k$ be the set of periodic sampling functions of period $p_k$, as described in Proposition \ref{cantorsubgroups}. Let $f\in P$ and $\epsilon>0$ be given. We can write $f=\sum_{j=1}^N a_j W_j, W_j\in P_j$. We also construct $s_0=\sum_{i=0}^Na_i^{(0)}W_i$ so that $\norm s_0\norm<\epsilon^2/72$ and $f_0=f+s_0$ is $p_N$-periodic and the corresponding spectrum has all $p_N$ gaps open. This is possible due to Lemma \ref{opengaps}, since we know that having gaps open is generic behavior.\\

Suppose that we have chosen $s_0,s_1,\ldots s_{k-1}$ and $f_0, f_1,\ldots, f_{k-1}$. Let $A_{k-1}$ be the minimal gap size of $\Sigma(f_{k-1})$ (we define gap size by absolute distance, rather than distance along the circular arc) and define $B_k=\min\{A_0,A_1,\ldots, A_{k-1}\}$. Applying Lemma \ref{opengaps}, we pick $s_k=\sum_{i=0}^{N+k}a_i^{(k)}W_i$ so that
\begin{align}
\norm s_k\norm<&\left(\frac{\epsilon}{2^k}\right)^2/72,\\
\norm s_k\norm<&\frac{1}{ 2^k}\frac{B_k^2}{3^2\cdot 72}\label{11},\\
f_k=&f+\sum_{j=0}^k s_j\text{ has all the gaps of its corresponding spectrum open. }
\end{align}
So the limit of $f_k$ exists, and we let $\tilde f=\lim_{k\to\infty} f_k$. By construction, we have $\norm\tilde f-f\norm <\epsilon^2/54$. We claim that $\Sigma(\mathcal E_{\tilde f})$ is nowhere dense. Equivalently, its complement is dense.\\

Given $z\in\Sigma(\mathcal E_{\tilde f})$ and $\tilde \epsilon>0$, we can pick $k$ large enough so
\begin{align}
\norm \tilde f-f_k\norm<&\left(\frac{\tilde \epsilon}{3}\right)^2/72\label{13},\\
 \frac{2\pi}{p_{N+k}}<&\frac{\tilde \epsilon}{3},\label{14}\\
\frac{\epsilon}{2^{k}}<&\frac{\tilde \epsilon}{3}\label{15}.
\end{align}
By (\ref{13}) and Proposition \ref{72}, we know that $\norm\mathcal E_{\tilde f}-\mathcal E_{f_k}\norm<\tilde \epsilon/3$, and so by Proposition \ref{rotation} there exists $z'\in \Sigma(\mathcal E_{f_k})$ such that $\vert z-z'\vert<\tilde\epsilon/3$. Moreover, by Lemma \ref{smallbands} and (\ref{14}) we can find $\tilde z$ in a gap of $\Sigma(\mathcal E_{f_k})$ such that $\vert z'-\tilde z\vert<\tilde\epsilon/3$. Write this gap of $\Sigma(f_k)$ that contains $\tilde z$ as $I_\delta(a)$, which refers to an open interval on the unit circle centered at the point $a\in\partial\D$, whose endpoints are each $\delta$ away from $a$ (this distance calculated in terms of absolute value, rather than arclength). By the triangle inequality, we have $2\delta>B_{k+1}$. We then know that, by (\ref{11}),
\[\norm \tilde f-f_k\norm =\left\vert\left\vert\sum_{j=k+1}^\infty s_j\right\vert\right\vert<\frac{B_{k+1}^2}{3^2\cdot 72}\left(\frac{1}{2^{k+1}}+\frac{1}{2^{k+2}}+\ldots\right)< \left(\frac{2\delta}{3}\right)^2\frac{1}{72}.\]
This implies $\norm \mathcal E_{\tilde f}-\mathcal E_{f_k}\norm<2\delta/3$.
So we have, by Proposition \ref{rotation} and the triangle inequality that $I_{\frac{\delta}{3}}(a)\cap\Sigma(\mathcal E_{\tilde f})=\emptyset$.\\
We claim that there exists $\delta'\in [\delta/3, \delta)$ such that $I_{\delta'}(a)\cap\Sigma( \mathcal E_{\tilde f})=\emptyset$ and $\vert\delta-\delta'\vert<\epsilon/2^{k}$. The above demonstrates that we may arrange for $\delta'\geq \delta/3$. Suppose that it is impossible to find such a $\delta'$ so that  $\vert\delta-\delta'\vert<\epsilon/2^{k}$. Then there will be an $x$ such that  $x \in\Sigma(\mathcal E_{\tilde f})$ and $I_{\frac{\epsilon}{2^{k}}}(x)\subseteq I_{\delta}(a)$. This contradicts the fact that $I_{\frac{\epsilon}{2^{k}}}(x)\cap \Sigma(\mathcal E_{f_k})\neq\emptyset$, since we had

\[\norm \tilde f-f_k\norm=\left\vert\left\vert \sum_{j={k+1}}^\infty s_j\right\vert\right\vert<\left(\frac{\epsilon}{2^k}\right)^2/72,\]
and so $\norm\mathcal E_{\tilde f}-\mathcal E_{f_k}\norm<\epsilon/2^k$.\\

Thus we can choose $\hat z$ in the gap of $\Sigma(\mathcal E_{\tilde f})$ that contains $I_{\delta'}(a)$ and so that $\vert\hat z-\tilde z\vert< \epsilon/2^{k}<\tilde \epsilon/3$. The second inequality follows from (\ref{15}). Since we also have $\vert\tilde z-z'\vert<\tilde \epsilon/3$ and $\vert z'-z\vert<\tilde\epsilon/3$, it follows that $\vert \hat z-z\vert<\tilde\epsilon$. This shows that $\partial\D\setminus\Sigma(\tilde f)$ is dense and completes the proof.
\end{proof}
\end{section}

\begin{section}{AC Spectrum}
{}
We now  prove Theorem \ref{t.ac}. The Floquet Theory results required in the proof are described in fuller detail in \cite{Simon}, Sections 11.1 and 11.2.

{}
Let $q$ be an even integer so that the Verblunsky coefficients are $q$-periodic. That is, if the period $p$ of the Verblunsky coefficients is odd, we take $q$ to be an even multiple. We shall consider the extended CMV matrix $\mathcal E$ as an operator acting on $\ell^{\infty}(\Z)$. The operator $\mathcal E$ is bounded on $\ell^\infty(\Z)$, since every row has only four nonzero terms that are bounded as long as $\alpha$ is bounded away from the boundary of the disk. Let $M$ be the shift operator $(Mu)_m=u_{m+q}$. We then have 
\[M\mathcal E=\mathcal EM.\]
Let $\Theta\in [0,2\pi)$ and define
\[ X_{\Theta}=\{ u\in \ell^\infty(\Z)\vert Mu=e^{i\Theta} u\}.\]
Note that $\mathcal E$ takes $X_\Theta$ to itself. Also, $X_\Theta$ is clearly $q$-dimensional, since it is determined by any $q$ consecutive coordinates of $u$. We then define $\mathcal E_q(\Theta)$ as the restriction of $\mathcal E$ to $X_\Theta$. The matrix $\mathcal E_q(\Theta)$ is described explicitly in Figure 11.3, (11.2.6), and (11.2.7) of \cite{Simon}. Given this machinery, let us now summarize Floquet Theory for OPUC.

\begin{prop}\label{Floquet}\
\begin{enumerate}[(a)]
\item We have $z\in \Sigma(\mathcal E)$ if and only if $zu=\mathcal E u$ for some solution $\{u(n)\}$ obeying $Mu=e^{i\Theta} u$ for some $e^{i\Theta}\in \partial\D$. In this case, $\tilde u=\left<u(n)\right>_{n=0}^{q-1}$ is an eigenvector of $\mathcal E_q(\Theta)$ corresponding to the eigenvalue $z$.
\item We have 
\[\Sigma(\mathcal E)=\bigcup_\Theta\Sigma(\mathcal E_q(\Theta)).\]
\item For $\Theta\neq 0,\pi$, we have 
\[ \det(z-\mathcal E_q(\Theta))=\left( \prod_{j=0}^{q-1}\rho_j\right) [z^{q/2} [\Delta(z)-(e^{i\Theta}+e^{-i\Theta})]],\]
where $\Delta(z)$ is the discriminant function defined in (11.1.2) of \cite{Simon}. Also, we know that 
\[\Sigma(\mathcal E)=\{z:\vert \Delta(z)\vert \leq 2\}.\]
The set $\Sigma(\mathcal E)$ is made of $q$ bands such that on each band, $\Sigma(\mathcal E)$ is either strictly
increasing or strictly decreasing.
\item If $z$ is on the boundary of some band, then $\Delta(z)=\pm 2$.
\end{enumerate}
\end{prop}
{}
We know that there is no point spectrum from Corollary \ref{c.gordon}. The discriminant $\Delta(z)$ is described in Section 11.1, and is explicitly defined in (11.1.2) in \cite{Simon}. Furthermore, again according to Theorem 11.1.2 of \cite{Simon} the spectrum of $\mathcal E$ consists precisely of the points $z$ on $\partial\D$  for which $\vert \Delta(z)\vert\leq 2$. 

{}
We shall first put the Fourier transform into a $\mod q$ setting, in a similar spirit as that in Section 5.3 of \cite{Simon2}. We define 
\[\mathcal F:\ell^2(\Z)\to L^2\left( \partial\D, \frac{d\Theta}{2\pi};\mathbb C^q\right),\]
the $L^2$ functions with values in $C^q$. The expression $\dfrac{d\Theta}{2\pi}$ refers to normalized Lebesgue measure on the unit circle. Thus given $n=0,\ldots ,q-1$:
\[ (\mathcal Fu)_n(\Theta)=\sum_{l=-\infty}^\infty u_{n+lq}e^{-il\Theta}.\]
We define this initially for $u\in \ell^1$ and then extend by using
\[\int_{\partial \D} \norm \mathcal F u_\cdot (\Theta)\norm^2 \frac{d\Theta}{2\pi} =\sum_n \vert u_n\vert^2,\]
since $\{ e^{il\Theta}\}_{l=-\infty}^\infty $ is a basis for $L^2(\partial \D, \dfrac{d\Theta}{2\pi})$. The inverse 
\[\mathcal F^{-1}: L^2\left(\partial \D, \frac{d\Theta}{2\pi}; \mathbb C^q\right)\to \ell^2(\mathbb Z),\]
is given by
\[(\mathcal F^{-1}f)_{n+lq}=\int e^{il\Theta} f_n(\Theta) \frac{d\Theta}{2\pi},\]
for $l\in \Z$ and $n=0,\ldots, q-1$.

{}
The function $\psi: z\to [0,\pi]$, defined in (11.2.19) of \cite{Simon} plays the role of $kp$ in \cite{Damanik-Gan}. That is, given any $z\in\Sigma(\mathcal E)$, we must have a solution $u=\phi$ to $\mathcal Eu=zu$ that satisfies $\phi_{n+q}=e^{i\psi(z)}\phi_n$. We also have an equilbrium measure $d\nu$ (denoted $d\rho$ in \cite{Damanik-Gan}) on the bands, that is written in the form $d\nu=V(\theta) d\theta$ where
\begin{equation}\label{bigV}
V(\theta)=\frac{1}{q\pi}\left\vert\frac{d\psi(e^{i\theta})}{d\theta}\right\vert.
\end{equation}
This is (11.2.25) in \cite{Simon}. 

{}
Given every $z\in\Sigma(\mathcal E)$, if we choose $\Theta=\psi(z)$ we know that $z$ is an eigenvalue of $\mathcal E_q(\Theta)$. Furthermore, from (11.2.19) in \cite{Simon} we can define $\psi(z)$ in terms of an $\arccos$ of a polynomial, and thus $\psi$ is smooth as a function on a band. In fact from (\ref{recurrence}) it is clear that there are two linearly independent solutions of $z \phi=\mathcal E\phi$. We label them $\phi^+(z),\phi^-(z)$ and assume without loss of generality that
\[\sum_{j=0}^{q-1} \vert \phi_j^\pm\vert^2=1.\]
 
For a vector $u$ of finite support, we define
\[V u^\pm(z)=\frac{q}{2}\sum_{n\in \Z}\overline{\phi_n^\pm(z)}u_n.\]

\begin{lemma}\label{unitary}
The $U=V\mathcal F$ operator extends to a unitary map of $\ell^2(\Z)$ to \\ $L^2(\Sigma(\mathcal E),d\nu(\theta);\mathbb C^2)$.
\end{lemma}
\begin{proof}
Note that $\tilde\phi^+\equiv\{\phi_n^+\}_{n=0}^{l-1}$ is a normalized eigenvector of $\mathcal E_q(\Theta)$. Thus if $\lambda_1,\ldots, \lambda_q$ are the eigenvalues for a certain $\mathcal E_q(\Theta)$, $\{\tilde \phi^+(\lambda_j)\}_{j=1}^q$ is an orthonormal basis for $\mathbb C^q$. Hence $V$ is just a unitary change of basis: this is made clear when we compute
\begin{align*}
d\nu=& V(\theta) d\theta,\\
=& \frac{1}{q\pi}\left\vert\frac{d\psi(e^{i\theta})}{d\theta}\right\vert d\theta,\\
=& \frac{1}{q\pi}\left\vert\frac{d\Theta}{d\theta}\right\vert d\theta,\\
=& \frac{2}{q} \frac{d\Theta}{2\pi}.
\end{align*}

Furthermore, since $\mathcal F$ is unitary as well $U$ must be also.
\end{proof}
We clearly also have
\[[U(\mathcal E u)]^{\pm}(z)=M_z[Uu]^{\pm}(z),\]
where $M_z$ is the multiplication by $z$ operator in $L^2(\Sigma(\mathcal E),d\nu(\theta);\mathbb C^2)$.

\begin{lemma}\label{gfunction}
Given the $q$-periodic sequence $\alpha$ and a vector $u$ of finite support, we can write the density function 
\begin{equation}\label{density}
g_{\alpha,u}=(\vert (U u)^+(e^{i\theta})\vert^2+\vert (Uu)^-(e^{i\theta})\vert^2)\left\vert\frac{d\nu({\theta})}{d\theta}\right\vert,
\end{equation}
so that the spectral measure associated with the sequence $\alpha$ and the vector $u$ is
\[d\mu_{\alpha,u}=g_{\alpha,u}(e^{i\theta}) d\theta.\]
\end{lemma}

\begin{proof}
Let $u$ be a finitely supported vector, and let $d\mu_u$ be the corresponding spectral measure of $\mathcal E$. 
Let $M_z$ be the multiplication by $z$ operator on $L^2(\Sigma(\mathcal E),d\mu_u)$. 
If $f$ is a polynomial, we have, clearly
\begin{align*}
\left<Uu,M_{f(z)}Uu\right>_{L^2(d\nu)}=&\left< Uu,Uf(\mathcal E)u\right>_{L^2(d\nu)},\\
=&\left< u,f(\mathcal E)u\right>,\\
=&\int_{\Sigma(\mathcal E)} f(e^{i\theta}) d\mu_u(\theta).
\end{align*}
This then clearly holds if we let $f$ be any $L^2$ function.
But we also know

\[\left< Uu, M_{f(z)}Uu\right>_{L^2(d\nu)}= \int_{\Sigma(\mathcal E)} f(e^{i\theta}) \norm Uu(e^{i\theta})\norm^2 d\nu(\theta).\]
\end{proof}
{}
With these correspondences, the proofs and statements of Lemmas 3.1-3.3 of \cite{Damanik-Gan} follow verbatim:

\begin{lemma}\label{3.1}
For every $t \in (1,2)$, there exists a constant $D =
D(\|\alpha\|_\infty,q,t)$ such that
\begin{equation}\label{rho}
\int_{\Sigma(\mathcal E)} \left| \frac{1}{q\pi}\frac{d\psi}{d\theta}(\theta) \right|^t \, d\theta \le D.
\end{equation}
\end{lemma}
\begin{proof}
By (11.2.23) of \cite{Simon} we have
$$
\left|  \frac{1}{q\pi}\frac{d\psi}{d\theta}(\theta) \right| = \left| \frac{\Delta' (e^{i\theta})}{2q\pi \,
\sin(\psi)} \right|.
$$
Since we can bound $|\Delta' (e^{i\theta})|$ by a
$(\|\alpha\|_\infty,q)$-dependent constant and \\
$\int_0^{\pi}(\sin(x))^{1-t} \, dx < \infty$, we have the following estimates, where $C_1,C_2$ are constants.
$$
\int_{\Sigma(\mathcal E)} \left|  \frac{1}{q\pi}\frac{d\psi}{d\theta}(\theta) \right|^t \, d\theta \leq C_1
\int_0^{\pi} \left| \frac{1}{2q \pi\, \sin(\psi)}
\right|^{t-1} \, d\psi \leq C_2 \int_0^\pi
|\sin(\psi)|^{1-t} \, d\psi,
$$
and the last integral may be bounded by a $t$-dependent constant.
\end{proof}

\begin{lemma}\label{l.one}
Let $u \in \ell^2(\mathbb{Z})$ have finite support. Then, for
every $t \in (1,2)$, there exists a constant $Q =
Q(\|\alpha\|_\infty,q,u,t)$ such that
\begin{equation}\label{Gexpre}
\int_{\Sigma(\mathcal E)} \left| g_{\alpha,u}(e^{i\theta}) \right|^t \, d\theta \le Q.
\end{equation}
\end{lemma}

\begin{proof}
Since $u$ has a finite support, we can find a constant $M =
M(q,u)$ such that $\norm[Uu](e^{i\theta})\norm^2 \leq M$. Thus, by (\ref{density})
we have
\begin{align*}
\int_{\Sigma(\mathcal E)} \left| g_{\alpha,u} (e^{i\theta}) \right|^t \, d\theta & =
\int_{\Sigma(\mathcal E)} \left[  \left( |U{u}^+(e^{i\theta})|^2 +
|U{u}^-(e^{i\theta})|^2 \right) \left| \frac{1}{q\pi}\frac{d\psi}{d\theta}(\theta) \right| \right]^t \, d\theta, \\
& \le M^t \int_{\Sigma(\mathcal E)} \left|
\frac{1}{q\pi}\frac{d\psi}{d\theta}(\theta) \right|^t
\, d\theta, \\
& \le  M^t D.
\end{align*}
with the constant $D$ from the previous lemma.
\end{proof}

\begin{lemma}\label{l.two}
Let $(X,d\mu)$ be a finite measure space, let $r>1$ and let $f_n,f
\in L^r$ with $\sup_n\| f_n \|_r < \infty$. Suppose that $f_n(x)
\rightarrow f(x)$ pointwise almost everywhere. Then, $\| f_n-f
\|_p \rightarrow 0$ for every $p < r$.
\end{lemma}

\begin{proof}
This is \cite[Lemma~2.6]{Avron-Simon}.
\end{proof}

\begin{lemma}\label{l.three}
Suppose $u \in \ell^2(\mathbb{Z})$ has finite support and $\alpha^{(n)},\alpha :
\Z \to \partial\D$ are $q$-periodic and such that $\| \alpha^{(n)} - \alpha \|_\infty
\rightarrow 0$ as $n \rightarrow \infty$. Then, for any $t \in
(1,2)$, we have
$$
\int_{\partial\D} \left|g_{\alpha^{(n)},u}(e^{i\theta}) - g_{\alpha,u}(e^{i\theta}) \right|^t \, d\theta
\rightarrow 0,
$$
as $n \to \infty$.
\end{lemma}
\begin{proof}
By Lemmas~\ref{l.one} and \ref{l.two} we only need to prove
pointwise convergence. Given the explicit identity \eqref{density},
pointwise convergence follows readily from the following two
facts: the discriminant of the approximants converges pointwise to
the discriminant of the limit and the matrices $\mathcal E({e^{i\psi}})$ defined in (11.2.4) and Figure 11.3 of \cite{Simon}  associated
with the approximants converge pointwise to those associated with
the limit and therefore so do the associated eigenvectors.
\end{proof}

\begin{proof}[Proof of Theorem \ref{t.ac}]
The idea is to modify the construction from the proof of
Theorem~\ref{t.cantor}. Thus, we will again start with an
arbitrarily small ball in $C(\Omega,\partial\D)$ and construct a point in
this ball for which the associated CMV operator has both
Cantor spectrum and purely absolutely continuous spectrum. The
presence of absolutely continuous spectrum then also implies that
the Lebesgue measure of the spectrum is positive.

Fix $t\in (1,2)$ and let $u \in \ell^2(\mathbb{Z})$ have finite
support. In going through the construction in the proof of
Theorem~\ref{t.cantor}, pick $s_k$ so that in addition to the
conditions above, we have
\begin{equation}\label{Gfinal}
\left(\int_{\partial\D} \left|g_u^{k-1}(e^{i\theta}) -
g_u^{k}(e^{i\theta})\right|^t d\theta\right )^{\frac{1}{t}} \leq \frac{1}{2^k},
\end{equation}
where $g_u^{k}$ is the density of the spectral measure associated
with $u$ and the periodic Verblunsky coefficient $n \mapsto f_k(T^n \omega)$,
with the estimate above being uniform in $\omega \in \Omega$. This
is possible due to Lemma~\ref{l.three}.

By Lemma~\ref{l.one}, there exists a constant $Q(u,t) < \infty$
such that $\int_\R \left|g_u^k(e^{i\theta})\right|^t \, d\theta \leq Q(u,t)$.

Now fix any $\omega \in \Omega$. Let $A$ be a finite union of open
sets in $\partial \D$. If $P_A^k$ is the spectral projection for the Verblunsky coefficients  $n
\mapsto f_k(T^n \omega)$ and $P_A$ is the spectral projection for
the Verblunsky coefficients $n \mapsto \tilde f(T^n \omega)$, it follows that
$\langle u,P_A u \rangle \le \limsup_{k\rightarrow \infty} \langle
u, P_A^{k} u \rangle$ since $\|f_k - \tilde{f}\|_\infty \to 0$
and hence the associated CMV operators converge in norm.

Applying H\"older's inequality, we find
$$
\langle u, P_A u \rangle \le \limsup_{k \to \infty} \int_A
g_u^k(e^{i\theta}) \, d\theta \le Q(u,t) |A|^{\frac{1}{q}},
$$
where $\frac{1}{q}+\frac{1}{t}=1$ and $| \cdot |$ denotes Lebesgue
measure. This shows that the spectral measure associated with $u$
and the CMV operator with potential $n \mapsto \tilde
f(T^n \omega)$ is absolutely continuous with respect to Lebesgue
measure. Since this holds for every finitely supported $u$, it
follows that this operator has purely absolutely continuous
spectrum.
\end{proof}

\end{section}

\bibliographystyle{unsrt}   
\bibliography{mybib}
\end{document}